\documentclass[11pt]{article}
\usepackage[margin=1in]{geometry}
\usepackage{amsmath,amssymb,amsthm,mathtools}
\usepackage{stmaryrd}
\usepackage[T1]{fontenc}
\usepackage{lmodern}
\usepackage{microtype}
\usepackage[hidelinks]{hyperref}

\DeclareMathOperator{\TV}{TV}
\DeclareMathOperator{\Ber}{Ber}
\DeclareMathOperator{\Mult}{Mult}

\newcommand{\Dirac}[1]{\left\llbracket #1 \right\rrbracket}

\newtheorem{theorem}{Theorem}
\newtheorem{lemma}[theorem]{Lemma}

\theoremstyle{remark}
\newtheorem{remark}[theorem]{Remark}

\newcommand{\E}{\mathbb E}
\newcommand{\R}{\mathbb R}

\newcommand{\dd}{\mathop{}\!\mathrm{d}}
\newcommand{\mathe}{\mathrm{e}}

\newcommand{\LamP}{\Lambda_P}
\newcommand{\LamQ}{\Lambda_Q}

\newcommand{\Vhet}{V}
\newcommand{\Vmf}{\bar V}
\newcommand{\Shet}{S}
\newcommand{\Smf}{\bar S}
\newcommand{\VY}{V}
\newcommand{\SY}{S}
\newcommand{\RY}{R}

\newcommand{\De}[1]{\tilde\Delta\paren{{#1}}}
\newcommand{\Dr}[1]{\Delta\paren{{#1}}}

\renewcommand{\vec}[1]{\boldsymbol{#1}}
\newcommand{\abs}[1]{\left| #1 \right|}
\newcommand{\paren}[1]{\left( #1 \right)}

\newcommand{\set}[1]{\left\{ #1 \right\}}

\newcommand{\beq}{\begin{eqnarray*}}
\newcommand{\eeq}{\end{eqnarray*}}
\newcommand{\beqn}{\begin{eqnarray}}
\newcommand{\eeqn}{\end{eqnarray}}

\title{
A homogenization principle for total variation
}
\author{Aryeh Kontorovich}
\date{\today}

\begin{document}
\maketitle

\begin{abstract}
We prove an inequality comparing 
the variational distance between
pairs of product probability measures 
to its homogenized counterpart.
If $
P_1,\ldots,P_n,
Q_1,\ldots,Q_n
$
are arbitrary probability
measures on a measurable space
and
$
\bar P:=\frac1n\sum_{i=1}^n P_i,
\bar Q:=\frac1n\sum_{i=1}^n Q_i
$, we show that
$$
\TV\!\left(\bigotimes_{i=1}^n P_i,
\bigotimes_{i=1}^n Q_i\right)
\;\ge\;
c\,
\TV(\bar P^{\otimes n},\bar Q^{\otimes n})
,
$$
where $c>0$ is a universal constant.

The proof is based on a one-dimensional representation of total variation between products. 
We embed pairs of probability distributions
$P_i,Q_i$
into positive measures $\eta_i$ on $\R$.
We then define a functional $T$ 
over measures on $\R$
that realizes
TV over products via convolution:
$
\TV\!\left(\bigotimes_{i=1}^n P_i,
\bigotimes_{i=1}^n Q_i\right)
=
T(\eta_1*\cdots *\eta_n).
$
Our main analytic discovery is
that for the relevant class of positive
measures $\eta_i$,
the convolution inequality
$
T(\eta_1*\cdots *\eta_n)
\ge
c\,
T\!\left(
\bar\eta
^{*n}
\right)
$
holds,
where
$
\bar\eta
=
\frac1n\sum_{i=1}^n \eta_i
$.
Finally, a higher-dimensional lifting argument
shows that
$
T\!\left(
\bar\eta
^{*n}
\right)
\ge
\TV(\bar P^{\otimes n},\bar Q^{\otimes n})
$.
To our knowledge, both the exact representation and the convolution inequality are new.
\end{abstract}

\section{Introduction}
For a product distribution
$\vec P=P_1\otimes 
\ldots\otimes P_n$,
its {\em homogenized} version is
$\bar{\vec P}
=\paren{\frac1n \sum_{i=1}^n P_i}^{\otimes n}$.
Since the latter is considerably more analytically
tractable than the former,
it is natural to ask how the heterogeneous product compares with its homogenized counterpart.
In particular,
whereas the total variation
distance between homogeneous
products is fairly well understood
---
it is known
\cite[Corollary 16.2]{Polyanskiy_Wu_2024}
that asymptotically,
$
\TV(P^{\otimes n},Q^{\otimes n})
\approx
1-\exp(-nC(P,Q))
$,
where
$C(P,Q)
=
-\inf_{\lambda\in[0,1]}\log\int(\dd P)^{1-\lambda}(\dd Q)^{\lambda}
$
is the {\em Chernoff information}
---
no such simple, analytically
tractable asymptotic is known
for the inhomogeneous case.

Our main result is an inequality comparing
the total variation 
distance between
two arbitrary product distributions
and their homogenized versions:
\beq
\TV(\vec P,\vec Q)
&\ge&
c\,
\TV(\bar{\vec P},\bar{\vec Q}),
\eeq
where $c>0.1489$ is an absolute constant;
this holds uniformly over all measurable spaces.
For the special case
of
Bernoulli measures (and a fortiori in general), 
it was shown 
in \cite{kontorovich2026tvhomogenization}
that $c\le\frac89$ and this value was conjectured to be optimal.
Understanding the exact optimal constant and the extremizing measures achieving it remains an intriguing open problem.
Choosing $P_1=P_2=\Ber(\frac12)$
and $Q_{1,2}=\Ber(\frac12\pm\varepsilon)$
shows that in general, no reverse inequality can hold.

The proof is based
on an exact one-dimensional representation of 
the
total
variation
distance over
product distributions. 
The analytic and combinatorial heart of the argument
deals with
strictly positive probability mass functions $P_i,Q_i$
on a finite set $\Omega$; after that,
the general case is a straightforward generalization.
For such $P_i,Q_i$,
define
\beqn
\label{eq:etai}
\eta_i:=
\sum_{\omega\in\Omega}
\sqrt{P_i(\omega)Q_i(\omega)}\,
\Dirac{\frac12\log\frac{P_i(\omega)}{Q_i(\omega)}},
\eeqn
where $\Dirac{u}
=\delta_u
$ is the 
Dirac measure associated with the point mass $u\in\R$.

We observe (and prove in Lemma~\ref{lem:admiss-closure} below) that the $\eta_i$ in \eqref{eq:etai}
satisfy
\beqn
\label{eq:admiss}
\int_{\R} \mathe^x\,\eta(\dd x)=1
\qquad\text{and}\qquad
\int_{\R} \mathe^{-x}\,\eta(\dd x)=1
\eeqn
and in general say that
a finitely supported positive measure $\eta$ on $\R$ is \emph{admissible} if it satisfies \eqref{eq:admiss}.
For an admissible measure $\eta$, define
\beqn
\label{eq:Tdef}
T(\eta) &=& \frac12\int_{\R}\abs{\mathe^x-\mathe^{-x}}\,\eta(\dd x).
\eeqn
If $\eta_1,\dots,\eta_n$ are finite positive measures on $\R$, write
\[
\eta_1*\cdots *\eta_n
\]
for their convolution, and write $\eta^{*n}$ for the $n$-fold convolution of
a single measure $\eta$.
We recall that on a measurable space
$(\Omega,\mathcal F)$, the variational distance between two probability measures $\mu,\nu$
is $\TV(\mu,\nu)=\sup_{E\in\mathcal F}|\mu(E)-\nu(E)|$
and for finite $\Omega$ is equivalent to
$\frac12\sum_{\omega\in\Omega}|\mu(\omega)-\nu(\omega)|$.

Our first structural result shows that this
encoding recovers total variation exactly and linearizes products:

\begin{theorem}
\label{thm:encoding}
Let $\Omega$ be a finite set,
$n\ge1$,
and $P_i,Q_i$, $i\in[n]$, strictly positive probability
mass functions on $\Omega$.
If $(P_i,Q_i)\mapsto\eta_i$
are defined as in \eqref{eq:etai}
and $T$ as in \eqref{eq:Tdef},
then 
\beqn
\label{eq:TV=T}
\TV\!\left(\bigotimes_{i=1}^n P_i,\bigotimes_{i=1}^n Q_i\right)
&=&
T(\eta_1*\cdots *\eta_n).
\eeqn
\end{theorem}

Our main analytic result
is the following convolution inequality:
\begin{theorem}
\label{thm:main-conv}
For every 
$n\ge 1$ and every 
finitely supported 
admissible family
$\eta_1,\dots,\eta_n$ on $\R$,
\beq
T\!\left(\left(\frac1n\sum_{i=1}^n \eta_i\right)^{*n}\right)
&\le&
C_0\,T(\eta_1*\cdots *\eta_n),
\eeq
where 
$
C_0<6.7129
$
is an absolute constant.
\end{theorem}

The final step is to
understand how the admissible encoding map
$(P_i,Q_i)\mapsto\eta_i$
interacts with homogenization.
Putting 
$
\bar\eta=\frac1n\sum_{i=1}^n \eta_i
$,
it would be naive to expect
$(\bar P,\bar Q)\mapsto\bar\eta$,
since the map is highly nonlinear.
We circumvent this obstacle
by interpreting homogenization
as a lifted probability measure
over $[n]\times\Omega$:

\begin{theorem}
\label{thm:lift}
Let $\Omega$ be a finite set,
$n\ge1$,
and $P_i,Q_i$, $i\in[n]$, strictly positive probability
mass functions on $\Omega$
and define 
$(P_i,Q_i)\mapsto\eta_i$
as in \eqref{eq:etai}.
Define probability mass functions on $[n]\times\Omega$ by
\[
\LamP(i,\omega):=\frac1n P_i(\omega),
\qquad
\LamQ(i,\omega):=\frac1n Q_i(\omega).
\]
Then the admissible encoding of the pair $(\LamP,\LamQ)$ in the sense of
\eqref{eq:etai}
is exactly
$
\bar\eta:=\frac1n\sum_{i=1}^n \eta_i.
$
Consequently,
\beq
T(\bar\eta^{*n})
&=&\TV(\LamP^{\otimes n},\LamQ^{\otimes n}).
\eeq
\end{theorem}

From here, the proof of
our homogenization principle is straightforward.

\begin{theorem}\label{thm:main-hom}
Let $(\Omega,\mathcal F)$ be a measurable space
and let
$P_1,\dots,P_n,\ Q_1,\dots,Q_n$
be probability measures on $(\Omega,\mathcal F)$. 
Define
\beq
\vec P:=\bigotimes_{i=1}^n P_i,
\qquad
\vec Q:=\bigotimes_{i=1}^n Q_i,
\qquad
\bar P:=\frac1n\sum_{i=1}^n P_i,
\qquad
\bar Q:=\frac1n\sum_{i=1}^n Q_i,
\qquad
\bar{\vec P}:=\bar P^{\otimes n},
\qquad
\bar{\vec Q}:=\bar Q^{\otimes n}.
\eeq
Then
\beq
\TV(\vec P,\vec Q)&\ge& c\,
\TV(\bar{\vec P},\bar{\vec Q})
,
\eeq
where $c>0.1489$ is an absolute constant.
For finite $\Omega$, the 
homogenized term
has a particularly simple expression
in terms of multinomials:
\beqn
\label{eq:mult}
\TV(\bar{\vec P},\bar{\vec Q})
&=&
\TV(\Mult(n,\bar P),\Mult(n,\bar Q))
.
\eeqn
\end{theorem}
\begin{proof}
We first assume that
$\Omega$ is finite and the $P_i, Q_i$
are strictly positive probability mass functions.

For each $i\in[n]$, define
$(P_i,Q_i)\mapsto\eta_i$
as in \eqref{eq:etai}.
Lemma~\ref{lem:admiss-closure}
implies
each $\eta_i$ is admissible and
Theorem~\ref{thm:encoding} that
$
\TV(\vec P,\vec Q)
=
T(\eta_1*\cdots *\eta_n)
.
$
Define probability mass functions
$\LamP,\LamQ$
on $[n]\times\Omega$ 
as in 
Theorem~\ref{thm:lift};
the latter implies that
$
T(\bar\eta^{*n})=\TV(\LamP^{\otimes n},\LamQ^{\otimes n})
$,
where
$
\bar\eta:=\frac1n\sum_{i=1}^n \eta_i
$.
Invoking Theorem~\ref{thm:main-conv},
we obtain
\beq
\TV(\LamP^{\otimes n},\LamQ^{\otimes n})
=
T(\bar\eta^{*n})
\le
C_0\,T(\eta_1*\cdots *\eta_n)
=
C_0\,\TV(\vec P,\vec Q),
\qquad
C_0<6.7129.
\eeq
Define the map
$\pi:[n]\times\Omega\to\Omega$
by $\pi(i,\omega)=\omega$.
We make two simple observations about $\pi$: first,
pushforwards satisfy
\[
\pi_{\#}\LamP=\sum_{i=1}^n \LamP(i,\cdot)=\frac1n\sum_{i=1}^n P_i=\bar P
\]
(and analogously for $Q$)
and secondly, they
induce
a Markov kernel
---
as does the 
product
map $
\pi^{\otimes n}
$.
Thus, the data processing inequality
\cite[Theorem 7.4]{Polyanskiy_Wu_2024}
applies:
\beq
\TV(\bar{\vec P},\bar{\vec Q})
&\le&
\TV(\LamP^{\otimes n},\LamQ^{\otimes n}).
\eeq
This completes the proof of
$
\TV(\bar{\vec P},\bar{\vec Q})
\le
C_0\,\TV(\vec P,\vec Q)
$
in the finite $\Omega$, strictly positive case.
The positivity assumption is removed via a standard continuity argument: one perturbs
an arbitrary $P_i$ (resp., $Q_i$)
via $P_i^{(\delta)}(\omega)=(1-\delta)P_i(\omega)+\frac{\delta}{|\Omega|}$
and takes $\delta\downarrow0$, noting 
that $\TV(\cdot,\cdot)$ is continuous in both arguments.

The extension to general measurable spaces 
is via a standard approximation argument,
spelled out in
Lemma~\ref{lem:rect} below:
For any probability measures $\mu,\nu$ on $(\Omega,\mathcal F)$ and any $n\ge 1$,
we have
$
\TV(\mu^{\otimes n},\nu^{\otimes n})
=
\sup_{\mathcal E}
\TV(\mu_{\mathcal E}^{\otimes n},\nu_{\mathcal E}^{\otimes n}),
$
where the supremum is over all finite $\mathcal F$-measurable partitions
$\mathcal E=\{E_1,\dots,E_m\}$ of $\Omega$, and
$\mu_{\mathcal E}(j)=\mu(E_j)$, $\nu_{\mathcal E}(j)=\nu(E_j)$.
Since 
the quotient map
$
\Omega\to[m]
$
induces a Markov kernel, we have
\beq
\TV(\vec P,\vec Q)
\;\ge\;
\sup_{\mathcal{E}}
\TV\!\left(
\bigotimes_{i=1}^n 
(P_i)_{\mathcal{E}}
,
\bigotimes_{i=1}^n (Q_i)_{\mathcal{E}}
\right)
\;\ge\;
c
\sup_{\mathcal{E}}
\TV\!\left(
\bar P_{\mathcal{E}}^{\otimes n}
,
\bar Q_{\mathcal{E}}^{\otimes n}
\right)
\;=\;
c
\TV\!\left(
\bar P^{\otimes n}
,
\bar Q^{\otimes n}
\right)
.
\eeq
Finally, \eqref{eq:mult} is a standard fact, which we prove in
Lemma~\ref{lem:mult} for completeness.
\end{proof}

\subsection*{Related work}

The most directly relevant prior
work is \cite{kontorovich2026tvhomogenization}, which proves 
$
\TV(\vec P,\vec Q)\ge c\,\TV(\bar{\vec P},\bar{\vec Q})
$
for the special case of 
$\Omega=\{0,1\}$
(i.e., products of Bernoullis)
and a worse constant ($c=0.0115$).
The proof techniques employed therein appear not to generalize to
general measurable spaces
and are quite distinct from the methods used here.

Regarding the convolution,
Roos obtained explicit total-variation bounds for heterogeneous convolution products on measurable Abelian groups and semigroups.
In particular,
\cite{Roos2010ClosenessConvolutions}
studies approximation by the 
$n$-fold convolution of the arithmetic mean, and 
\cite{Roos2017RefinedTV}
refines related nonasymptotic bounds in multivariate and compound Poisson approximation.

Conceptually, our construction fits in the classical
Hellinger--Bhattacharyya--Kakutani line of ideas in that it starts from the
overlap measure $\sqrt{PQ}$ and the log-likelihood ratio
\cite{Hellinger1909,Bhattacharyya1946,Kakutani1948,LeCamYang2000,Torgersen1991}.
The relevance of Kakutani's
result is structural:
it identifies Hellinger-type overlap as a canonical score for product
measures.
In \cite{Feng24Deterministically}, the total variation between
product measures is recovered from a one-dimensional distribution of the
likelihood ratio. Classically, 
\cite{Birge1983} studied non-i.i.d.\ testing through the Hellinger geometry
of product measures, replacing total variation by the tensorized Hellinger
metric.

A second recurring theme is the bounded transform
$
\frac{u-v}{u+v}
=
\tanh\!\left(\frac12\log\frac{u}{v}\right),
$
which has appeared in robust testing and Hellinger-based procedures
\cite{Huber1965,Baraud2011,BaraudBirge2018,Suresh2021}. In that literature, such
quantities serve as bounded surrogates for the log-likelihood ratio. Here the
same transform arises from an exact representation theorem and becomes the basic
analytic variable in the convolution inequality.
The representation
$\min\{u,v\}=\sqrt{uv}\exp\left(
-\frac12\left|\log\frac u v\right|
\right)$, reminiscent of our convolutional encoding \eqref{eq:etai}, was used in
\cite{kon25tens} to prove
$
\TV(\vec P,\vec Q)\ge 
c\min\{1,
\sqrt{
\sum_{i=1}^n\TV(P_i,Q_i)^2
}
\}
$. As discussed in \cite{kontorovich2026tvhomogenization}, the homogenized and the $\ell_2$ lower bounds are in general incomparable.

\section{Proofs}

\subsection{Theorems \ref{thm:encoding}
and
\ref{thm:lift}
}

\begin{proof}[Proof of Theorem \ref{thm:encoding}]
We begin with a single coordinate $i$:
\begin{align*}
T(\eta_{i})
&=
\frac12
\sum_{\omega\in\Omega}
\left|
\mathe^{\frac12
\log\frac{P_i(\omega)}{Q_i(\omega)}
}
-
\mathe^{-\frac12
\log\frac{P_i(\omega)}{Q_i(\omega)}
}
\right|
\sqrt{P_i(\omega)Q_i(\omega)}
\\
&=
\frac12
\sum_{\omega\in\Omega}
\left|
\sqrt{
\frac{P_i(\omega)}{Q_i(\omega)}
}
-
\sqrt{
\frac{Q_i(\omega)}{P_i(\omega)}
}
\right|
\sqrt{P_i(\omega)Q_i(\omega)}
\\
&=
\frac12
\sum_{\omega\in\Omega}\abs{P_i(\omega)-Q_i(\omega)}
=
\TV(P_i,Q_i).
\end{align*}
To 
prove 
the claim for products, 
let
$\vec\omega=(\omega_1,\dots,\omega_n)\in\Omega^n$
and write
$
L(\vec\omega):=\sum_{i=1}^n \frac12\log\frac{P_i(\omega_i)}{Q_i(\omega_i)}.
$
By definition of convolution, $\eta_1*\cdots *\eta_n$ is the pushforward of
$\bigotimes_{i=1}^n \eta_i$ under addition, and therefore
\[
\eta_1*\cdots *\eta_n
=
\sum_{\vec\omega}
\left(\prod_{i=1}^n \sqrt{P_i(\omega_i)Q_i(\omega_i)}\right)
\Dirac{L(\vec\omega)},
\]
where repeated atoms are combined. Hence
\begin{align*}
T(\eta_1*\cdots *\eta_n)
&=
\frac12
\sum_{\vec\omega}
\abs{\mathe^{L(\vec\omega)}-\mathe^{-L(\vec\omega)}}
\prod_{i=1}^n \sqrt{P_i(\omega_i)Q_i(\omega_i)}\\
&=
\frac12
\sum_{\vec\omega}
\abs{\prod_{i=1}^n P_i(\omega_i)-\prod_{i=1}^n Q_i(\omega_i)}\\
&=
\TV\!\left(\bigotimes_{i=1}^n P_i,\bigotimes_{i=1}^n Q_i\right).
\end{align*}
\end{proof}

\begin{proof}[Proof of Theorem \ref{thm:lift}]
By Theorem~\ref{thm:encoding}, the admissible encoding of the pair
$(\LamP,\LamQ)$ is
\[
\sum_{i=1}^n \sum_{\omega\in\Omega}
\sqrt{\LamP(i,\omega)\LamQ(i,\omega)}\,
\Dirac{\frac12\log\frac{\LamP(i,\omega)}{\LamQ(i,\omega)}}.
\]
Now
\[
\sqrt{\LamP(i,\omega)\LamQ(i,\omega)}
=
\sqrt{\frac1n P_i(\omega)\cdot \frac1n Q_i(\omega)}
=
\frac1n\sqrt{P_i(\omega)Q_i(\omega)},
\]
and
\[
\frac12\log\frac{\LamP(i,\omega)}{\LamQ(i,\omega)}
=
\frac12\log\frac{(1/n)P_i(\omega)}{(1/n)Q_i(\omega)}
=
\frac12\log\frac{P_i(\omega)}{Q_i(\omega)}.
\]
Therefore the encoding becomes
\[
\sum_{i=1}^n \sum_{\omega\in\Omega}
\frac1n\sqrt{P_i(\omega)Q_i(\omega)}\,
\Dirac{\frac12\log\frac{P_i(\omega)}{Q_i(\omega)}}
=
\frac1n\sum_{i=1}^n \eta_i
=
\bar\eta.
\]
The final claim follows from
Theorem~\ref{thm:encoding}
applied to the $n$-fold pair $(\LamP,\LamQ)$.
\end{proof}

\subsection{Proof of Theorem \ref{thm:main-conv}}

We restate the theorem in a way that facilitates computing explicit constants.
\begin{theorem}\label{thm:conv-w-const}
For $\varepsilon>0$, define
\[
\De{\varepsilon}
:=
\sqrt{
(1+6\varepsilon)\,
\frac{\sinh(2\varepsilon)-2\varepsilon}{(2\varepsilon)^2}
}
\]
and
\[
C(\varepsilon)
:=
\max\left\{
\frac{4\sqrt 2+\De{\varepsilon}}{1-\De{\varepsilon}},
\sqrt{\frac{1+\mathe^{-\varepsilon}}{1-\mathe^{-\varepsilon}}}
\right\}.
\]
Let
\[
C_0:=\inf\bigl\{C(\varepsilon):\ \varepsilon>0,\ \De{\varepsilon}<1\bigr\}.
\]
Then for every integer $n\ge 1$ and every finitely supported admissible family
$\eta_1,\dots,\eta_n$ on $\R$,
\[
T\!\left(\left(\frac1n\sum_{i=1}^n \eta_i\right)^{*n}\right)
\le
C_0\,T(\eta_1*\cdots *\eta_n).
\]
\end{theorem}

\subsubsection*{High-level proof outline}

The proof has two conceptual steps and two analytic regimes. The structural step,
carried out in Lemma~\ref{lem:score}, rewrites the functional $T$ as the
expectation of an explicit multilinear form $\Psi$ of independent centered
bounded random variables. Under this representation, replacing a heterogeneous
family $\eta_1,\dots,\eta_n$ by its arithmetic mean $\bar\eta$ corresponds
exactly to replacing the heterogeneous variables by i.i.d.\ variables with the
averaged one-coordinate law.

The analytic step is to compare the heterogeneous and homogenized evaluations of
$\Psi$. This comparison is governed by the mass defect
\[
\alpha:=\sum_{i=1}^n \bigl(1-\eta_i(\R)\bigr).
\]
When $\alpha$ is small, the multilinear form $\Psi$ is well approximated by its
linear part, and the problem reduces to square-function estimates together with
a Laplace-transform ordering; this is the content of
Lemmas~\ref{lem:signal}--\ref{lem:laplaceorder}. When $\alpha$ is large, the
functional $T$ is controlled directly by the total mass of the admissible
measures; this is captured by Lemma~\ref{lem:masscontrol}. We proceed to lay down the
ingredients needed for these two regimes.

\begin{lemma}[Multilinear score representation]\label{lem:score}
Let $\eta_1,\dots,\eta_n$ be finitely supported admissible measures on $\R$.
For each $i$, define a probability measure
\[
M_i(\dd x):=\cosh(x)\,\eta_i(\dd x),
\]
let $X_i\sim M_i$ independently, and set
\[
U_i:=\tanh(X_i)\in[-1,1].
\]
Define
\[
\Psi(y_1,\dots,y_n):=\frac12\left(\prod_{i=1}^n(1+y_i)-\prod_{i=1}^n(1-y_i)\right).
\]
Then $\E U_i=0$ for every $i$, and
\[
T(\eta_1*\cdots *\eta_n)=\E\big|\Psi(U_1,\dots,U_n)\big|.
\]
Moreover, if
$(\mathrm{a})$
$\bar\eta:=\frac1n\sum_{i=1}^n \eta_i$,
$(\mathrm{b})$
\[
\bar M(\dd x):=\cosh(x)\,\bar\eta(\dd x)=\frac1n\sum_{i=1}^n M_i(\dd x),
\]
$(\mathrm{c})$
$\bar X\sim\bar M$, 
$(\mathrm{d})$
$\bar U:=\tanh(\bar X)$, 
and 
$(\mathrm{e})$
$\bar U_1,\dots,\bar U_n$ are i.i.d.\ copies of $\bar U$, then
\[
T(\bar\eta^{*n})=\E\big|\Psi(\bar U_1,\dots,\bar U_n)\big|,
\]
and the law of $\bar U$ is the arithmetic average of the laws of
$U_1,\dots,U_n$.
\end{lemma}

\begin{proof}
Since $\eta_i$ is admissible,
\[
M_i(\R)=\int \cosh(x)\,\eta_i(\dd x)
=\frac12\int (\mathe^x+\mathe^{-x})\,\eta_i(\dd x)=1,
\]
so $M_i$ is a probability measure. Also,
\[
\E U_i
=
\int \tanh(x)\,M_i(\dd x)
=
\int \sinh(x)\,\eta_i(\dd x)
=
\frac12\int (\mathe^x-\mathe^{-x})\,\eta_i(\dd x)=0.
\]
Now use
\[
1\pm\tanh x=\frac{\mathe^{\pm x}}{\cosh x}.
\]
Since
\[
\eta_i(\dd x)=\operatorname{sech}(x)M_i(\dd x)=\cosh(x)^{-1}M_i(\dd x),
\]
we obtain
\begin{align*}
T(\eta_1*\cdots *\eta_n)
&=
\frac12\int_{\R^n}
\abs{\mathe^{x_1+\cdots+x_n}-\mathe^{-(x_1+\cdots+x_n)}}
\prod_{i=1}^n \eta_i(\dd x_i)\\
&=
\frac12\int_{\R^n}
\abs{
\prod_{i=1}^n \frac{\mathe^{x_i}}{\cosh x_i}
-
\prod_{i=1}^n \frac{\mathe^{-x_i}}{\cosh x_i}
}
\prod_{i=1}^n M_i(\dd x_i)\\
&=
\frac12\int_{\R^n}
\abs{
\prod_{i=1}^n (1+\tanh x_i)
-
\prod_{i=1}^n (1-\tanh x_i)
}
\prod_{i=1}^n M_i(\dd x_i)\\
&=
\E\big|\Psi(U_1,\dots,U_n)\big|.
\end{align*}
The same computation with $\bar\eta$ in place of $(\eta_i)$ gives
\[
T(\bar\eta^{*n})=\E\big|\Psi(\bar U_1,\dots,\bar U_n)\big|.
\]
Finally, for every Borel set $A\subseteq[-1,1]$,
\[
\mathbb P(\bar U\in A)
=
\bar M\bigl(\set{x:\tanh x\in A}\bigr)
=
\frac1n\sum_{i=1}^n M_i\bigl(\set{x:\tanh x\in A}\bigr),
\]
which says exactly that the law of $\bar U$ is the arithmetic average of the
laws of the $U_i$.
\end{proof}

\begin{lemma}[Mass defect and quadratic signal size]\label{lem:signal}
With the notation of Lemma~\ref{lem:score}, define
\[
m_i:=\eta_i(\R),
\qquad
\varepsilon_i:=1-m_i,
\qquad
\alpha:=\sum_{i=1}^n \varepsilon_i,
\qquad
\nu:=\sum_{i=1}^n \E U_i^2.
\]
Then $0\le m_i\le 1$ and
\[
\alpha\le \nu\le 2\alpha.
\]
\end{lemma}

\begin{proof}
Because $\cosh x\ge 1$ and $M_i(\R)=1$,
\[
m_i=\eta_i(\R)\le \int \cosh(x)\,\eta_i(\dd x)=1.
\]
Also,
\[
m_i
=
\int \operatorname{sech}(x)\,M_i(\dd x)
=
\E\sqrt{1-U_i^2},
\]
because $U_i=\tanh(X_i)$ and $\sqrt{1-\tanh^2 x}=\operatorname{sech}x$.
Therefore
\[
\varepsilon_i=\E\left[1-\sqrt{1-U_i^2}\right].
\]
For every $u\in[-1,1]$ one has
\[
\frac12u^2\le 1-\sqrt{1-u^2}\le u^2.
\]
Indeed, the right inequality is equivalent to
\[
\sqrt{1-u^2}\ge 1-u^2,
\]
which is obvious, and the left inequality is equivalent to
\[
\sqrt{1-u^2}\le 1-\frac12u^2.
\]
Since $1-\frac12u^2\ge \frac12>0$, we may square both sides, obtaining
\[
1-u^2\le 1-u^2+\frac14u^4.
\]
Taking expectation and summing over $i$ gives
\[
\frac12\nu\le \alpha\le \nu,
\]
which proves the claim.
\end{proof}

\begin{lemma}[Linearization of $\Psi$]\label{lem:linearization}
Let $Y_1,\dots,Y_n$ be independent centered random variables taking values in
$[-1,1]$, and define
\[
\rho:=\sum_{i=1}^n \E Y_i^2,
\qquad
S:=\sum_{i=1}^n Y_i,
\qquad
\Psi(Y_1,\dots,Y_n)=S+R.
\]
Then
\begin{itemize}
\item[$\mathrm{(a)}$] 
\[
\E R^2\le \sinh(\rho)-\rho,
\qquad
\E|R|\le \sqrt{\sinh(\rho)-\rho},
\]
\item[$\mathrm{(b)}$]
\[
\E|S|\ge \frac{\rho}{\sqrt{1+3\rho}},
\]
\item[$\mathrm{(c)}$]
$\Dr{\rho}$, defined by $\De{\rho}=\Dr{2\rho}$, i.e.,
\[
\Dr{\rho}:=
\begin{cases}
\sqrt{\dfrac{(1+3\rho)(\sinh \rho-\rho)}{\rho^2}}, & \rho>0,\\
0, & \rho=0,
\end{cases}
\]
is increasing on $[0,\infty)$, and
\[
(1-\Dr{\rho})\,\E|S|
\le
\E\big|\Psi(Y_1,\dots,Y_n)\big|
\le
(1+\Dr{\rho})\,\E|S|.
\]
\end{itemize}
\end{lemma}

\begin{proof}
Expanding the products gives
\[
\Psi(y_1,\dots,y_n)
=
\sum_{\substack{I\subseteq[n]\\ |I|\text{ odd}}}\prod_{i\in I} y_i.
\]
Therefore
\[
R=
\sum_{\substack{ I \subseteq[n]\\ | I |\ge 3,\ | I |\text{ odd}}}
\prod_{i\in I } Y_i.
\]
If $ I \neq  J$, 
the existence of an
$i\in  I \triangle  J$, together with independence
and centering,
implies
\[
\E\left[\prod_{i\in I }Y_i\prod_{j\in J}Y_j\right]=0.
\]
Thus distinct square-free monomials are orthogonal in $L^2$, and hence
\[
\E R^2
=
\sum_{\substack{ I \subseteq[n]\\ | I |\ge 3,\ | I |\text{ odd}}}
\prod_{i\in I }\E Y_i^2.
\]
Writing $a_i:=\E Y_i^2\ge 0$, we have $\sum_i a_i=\rho$, and for each $k\ge 1$,
\[
\sum_{1\le i_1<\cdots<i_k\le n}\prod_{\ell=1}^k a_{i_\ell}
\le \frac{\rho^k}{k!}.
\]
Therefore
\[
\E R^2
\le
\sum_{\substack{k\ge 3\\ k\text{ odd}}}\frac{\rho^k}{k!}
=
\sinh(\rho)-\rho.
\]
This yields
\[
\E|R|\le (\E R^2)^{1/2}\le \sqrt{\sinh(\rho)-\rho},
\]
proving $\mathrm{(a)}$.
Next,
\[
\E S^4
=
\sum_{i=1}^n \E Y_i^4
+
6\sum_{1\le i<j\le n} (\E Y_i^2)(\E Y_j^2).
\]
Since $\abs{Y_i}\le 1$, we have $Y_i^4\le Y_i^2$, and so
\[
\E S^4\le \rho+3\rho^2.
\]
Also, by H\"older's inequality,
\[
\E S^2
=
\E\bigl(\abs{S}^{2/3}\abs{S}^{4/3}\bigr)
\le
(\E\abs{S})^{2/3}(\E\abs{S}^4)^{1/3}
=
(\E\abs{S})^{2/3}(\E S^4)^{1/3}.
\]
If $\rho=0$, then $Y_i=0$ almost surely for every $i$, hence $S=0$ almost
surely, and the claimed lower bound for $\E\abs{S}$ is trivial. Assume now that
$\rho>0$. Since $\E S^2=\rho$, we obtain
\[
\rho\le (\E\abs{S})^{2/3}(\rho+3\rho^2)^{1/3},
\]
and therefore
\[
\E\abs{S}\ge \frac{\rho^{3/2}}{\sqrt{\rho+3\rho^2}}
=
\frac{\rho}{\sqrt{1+3\rho}},
\]
proving $\mathrm{(b)}$.
Together, (a) and (b) imply
\[
\frac{\E\abs{R}}{\E\abs{S}}
\le
\sqrt{\frac{(1+3\rho)(\sinh \rho-\rho)}{\rho^2}}
=
\Dr{\rho}
\qquad (\rho>0),
\]
and this is also trivial when $\rho=0$. Hence
\[
\E\abs{\Psi}
=
\E\abs{S+R}
\ge
\E\abs{S}-\E\abs{R}
\ge
(1-\Dr{\rho})\E\abs{S},
\]
and similarly
\[
\E\abs{\Psi}
\le
\E\abs{S}+\E\abs{R}
\le
(1+\Dr{\rho})\E\abs{S}.
\]
Finally, for $\rho>0$,
\[
\Dr{\rho}^2
=
(1+3\rho)\frac{\sinh \rho-\rho}{\rho^2}
=
(1+3\rho)\sum_{j=1}^{\infty}\frac{\rho^{2j-1}}{(2j+1)!},
\]
whose power-series coefficients are all nonnegative. Therefore $\Dr{\cdot}$ is
increasing on $(0,\infty)$, and since $\Dr{0}=0$, it is increasing on
$[0,\infty)$ as well.
\end{proof}

\begin{lemma}[Khintchine-type estimate]\label{lem:squarefunction}
Let $Y_1,\dots,Y_n$ be independent centered square-integrable real random
variables, and define
\[
\SY:=\sum_{i=1}^n Y_i,
\qquad
\VY:=\sum_{i=1}^n Y_i^2.
\]
Then
\[
\frac{1}{2\sqrt 2}\,\E\sqrt{\VY}
\le
\E\abs{\SY}
\le
2\,\E\sqrt{\VY}.
\]
\end{lemma}

\begin{proof}
Let $Y_1',\dots,Y_n'$ be an independent copy of $Y_1,\dots,Y_n$, and let
$\varepsilon_1,\dots,\varepsilon_n$ be independent Rademacher signs,
independent of everything else. Set
\[
\RY:=\sum_{i=1}^n \varepsilon_i Y_i.
\]
As in the usual symmetrization argument,
\[
\E\abs{\SY}
=
\E\left|\E'\sum_{i=1}^n (Y_i-Y_i')\right|
\le
\E\left|\sum_{i=1}^n (Y_i-Y_i')\right|
=
\E\left|\sum_{i=1}^n \varepsilon_i(Y_i-Y_i')\right|
\le
2\,\E\abs{\RY},
\]
while
\[
\E\abs{\RY}
=
\E\left|\E'\sum_{i=1}^n \varepsilon_i(Y_i-Y_i')\right|
\le
\E\left|\sum_{i=1}^n \varepsilon_i(Y_i-Y_i')\right|
=
\E\left|\sum_{i=1}^n (Y_i-Y_i')\right|
\le
2\,\E\abs{\SY}.
\]
Hence
\beqn
\label{eq:sym}
\frac12\,\E\abs{\RY}\le \E\abs{\SY}\le 2\,\E\abs{\RY}.
\eeqn
Conditioning on $Y=(Y_1,\dots,Y_n)$ and invoking the $p=1$ Khintchine inequality,
\[
\frac1{\sqrt2}\left(\sum_{i=1}^n Y_i^2\right)^{1/2}
\le
\E_{\varepsilon}\abs{\RY}
\le
\left(\sum_{i=1}^n Y_i^2\right)^{1/2};
\]
hence,
\[
\frac1{\sqrt2}\sqrt{\VY}\le \E_{\varepsilon}\abs{\RY}\le \sqrt{\VY}.
\]
Taking expectations over $Y$ gives
\[
\frac1{\sqrt2}\,\E\sqrt{\VY}\le \E\abs{\RY}\le \E\sqrt{\VY}.
\]
Combining this with \eqref{eq:sym} proves the claim.
\end{proof}

\begin{lemma}[Laplace-transform ordering]\label{lem:laplaceorder}
With the notation of Lemma~\ref{lem:score}, define
\[
\Vhet:=\sum_{i=1}^n U_i^2,
\qquad
\Vmf:=\sum_{i=1}^n \bar U_i^2.
\]
Then
\[
\E\sqrt{\Vmf}\le \E\sqrt{\Vhet}.
\]
\end{lemma}

\begin{proof}
Fix $\lambda>0$. Since the law of $\bar U$ is the arithmetic average of the laws
of the $U_i$, we have
\[
\E \mathe^{-\lambda \bar U^2}
=
\frac1n\sum_{i=1}^n \E \mathe^{-\lambda U_i^2}.
\]
Therefore, by AM--GM,
\beqn
\label{eq:laplace-AM-GM}
\E \mathe^{-\lambda \Vmf}
=
\left(\E \mathe^{-\lambda \bar U^2}\right)^n
=
\left(\frac1n\sum_{i=1}^n \E \mathe^{-\lambda U_i^2}\right)^n
\ge
\prod_{i=1}^n \E \mathe^{-\lambda U_i^2}
=
\E \mathe^{-\lambda \Vhet}.
\eeqn
Now use the standard identity, easily verified by integration by parts:
\[
\sqrt{x}
=
\frac1{2\sqrt{\pi}}
\int_0^{\infty}(1-\mathe^{-\lambda x})\,\lambda^{-3/2}\,\dd\lambda,
\qquad x\ge 0.
\]
Since the integrand is nonnegative, Tonelli's theorem applies. Therefore,
applying the identity to $\Vmf$ and $\Vhet$ and using the Laplace-transform 
comparison 
\eqref{eq:laplace-AM-GM}
proves the claim.
\end{proof}

\begin{lemma}[Mass estimates for admissible measures]\label{lem:masscontrol}
Let $\mu$ be admissible and write $m:=\mu(\R)$. Then
\[
1-m\le T(\mu)\le \sqrt{1-m^2}.
\]
\end{lemma}

\begin{proof}
Since $\mu$ is admissible,
\[
\int \cosh(x)\,\mu(\dd x)=1.
\]
Also,
\[
T(\mu)=\int \abs{\sinh x}\,\mu(\dd x).
\]
For the lower bound, use
\[
\abs{\sinh x}=\cosh x-\mathe^{-\abs{x}}
\]
to obtain
\[
T(\mu)=1-\int \mathe^{-\abs{x}}\,\mu(\dd x)\ge 1-\mu(\R)=1-m.
\]
For the upper bound,
\[
T(\mu)
=
\frac12\int \abs{\mathe^x-\mathe^{-x}}\,\mu(\dd x)
=
\frac12\int
\abs{\mathe^{x/2}-\mathe^{-x/2}}
(\mathe^{x/2}+\mathe^{-x/2})\,\mu(\dd x).
\]
By Cauchy--Schwarz,
\[
T(\mu)^2
\le
\frac14
\left(\int (\mathe^{x/2}-\mathe^{-x/2})^2\,\mu(\dd x)\right)
\left(\int (\mathe^{x/2}+\mathe^{-x/2})^2\,\mu(\dd x)\right).
\]
Now
\[
\int (\mathe^{x/2}-\mathe^{-x/2})^2\,\mu(\dd x)
=
\int (\mathe^x+\mathe^{-x}-2)\,\mu(\dd x)
=
2-2m,
\]
and
\[
\int (\mathe^{x/2}+\mathe^{-x/2})^2\,\mu(\dd x)
=
\int (\mathe^x+\mathe^{-x}+2)\,\mu(\dd x)
=
2+2m.
\]
Hence
\[
T(\mu)^2\le (1-m)(1+m)=1-m^2.
\]
This proves the lemma.
\end{proof}

\begin{proof}[Proof of Theorem~\ref{thm:conv-w-const}]
Choose any $\varepsilon>0$ 
such that 
$\De{\varepsilon}<1$;
Lemma \ref{lem:linearization}(c)
provides an interval of such choices.
Let
$\eta_1,\dots,\eta_n$ be finitely supported admissible measures on $\R$
and put
$\bar\eta:=\frac1n\sum_{i=1}^n \eta_i$.
Lemma~\ref{lem:admiss-closure}
implies that 
$\bar\eta$,
$\bar\eta^{*n}$,
and
$\eta_1*\cdots *\eta_n$
are all admissible.
Let $U_i$, $\bar U_i$, $\alpha$, and $\nu$ be as in
Lemmas~\ref{lem:score} and~\ref{lem:signal}, and write
\[
\Shet:=\sum_{i=1}^n U_i,
\qquad
\Smf:=\sum_{i=1}^n \bar U_i.
\]
By Lemma~\ref{lem:signal},
\[
\alpha\le \nu\le 2\alpha.
\]
We consider the two regimes.
\paragraph{Case I: $\alpha\le \varepsilon$.}
We have
$
\nu\le 2\alpha\le 2\varepsilon
$
and
also, because the law of $\bar U$ is the average of the laws of the $U_i$,
\[
n\,\E \bar U^2=\sum_{i=1}^n \E U_i^2=\nu.
\]
Set
\[
b:=\sqrt{\sinh(\nu)-\nu}.
\]
Applying Lemma~\ref{lem:linearization} to $(U_1,\dots,U_n)$ and to
$(\bar U_1,\dots,\bar U_n)$, and using Lemma~\ref{lem:score}, we obtain
\[
T(\eta_1*\cdots *\eta_n)
=
\E\abs{\Psi(U_1,\dots,U_n)}
\ge
\E\abs{\Shet}-b,
\]
\[
T(\bar\eta^{*n})
=
\E\abs{\Psi(\bar U_1,\dots,\bar U_n)}
\le
\E\abs{\Smf}+b.
\]
Furthermore, by Lemma~\ref{lem:linearization},
\[
b\le \Dr{\nu}\E\abs{\Shet}
\le
\Dr{2\varepsilon}\E\abs{\Shet}
=
\De{\varepsilon}\E\abs{\Shet}.
\]
Therefore
\[
T(\eta_1*\cdots *\eta_n)\ge (1-\De{\varepsilon})\,\E\abs{\Shet}.
\]
Now Lemmas~\ref{lem:squarefunction} and~\ref{lem:laplaceorder} give
\[
\E\abs{\Smf}\le 2\E\sqrt{\Vmf},
\qquad
\E\sqrt{\Vhet}\le 2\sqrt 2\,\E\abs{\Shet},
\qquad
\E\sqrt{\Vmf}\le \E\sqrt{\Vhet}.
\]
Combining these inequalities, 
\[
\E\abs{\Smf}\le 4\sqrt 2\,\E\abs{\Shet},
\]
whence
\[
T(\bar\eta^{*n})
\le
(4\sqrt 2+\De{\varepsilon})\,\E\abs{\Shet}
\le
\frac{4\sqrt 2+\De{\varepsilon}}{1-\De{\varepsilon}}
\,T(\eta_1*\cdots *\eta_n).
\]

\paragraph{Case II: $\alpha> \varepsilon$.}
Define
\[
M:=\bar\eta(\R)^n=\left(1-\frac{\alpha}{n}\right)^n;
\]
then
$
M\le \mathe^{-\alpha}\le \mathe^{-\varepsilon}.
$
Also, writing $m_i:=\eta_i(\R)=1-\varepsilon_i$, AM--GM gives
\[
(\eta_1*\cdots *\eta_n)(\R)
=
\prod_{i=1}^n m_i
\le
\left(\frac1n\sum_{i=1}^n m_i\right)^n
=
M.
\]
Applying Lemma~\ref{lem:masscontrol} to $\bar\eta^{*n}$ and
$\eta_1*\cdots *\eta_n$ yields
\[
T(\bar\eta^{*n})\le \sqrt{1-M^2},
\qquad
T(\eta_1*\cdots *\eta_n)
\ge
1-(\eta_1*\cdots *\eta_n)(\R)
\ge
1-M.
\]
Therefore
\[
T(\bar\eta^{*n})
\le
\frac{\sqrt{1-M^2}}{1-M}\,T(\eta_1*\cdots *\eta_n)
=
\sqrt{\frac{1+M}{1-M}}\,T(\eta_1*\cdots *\eta_n)
\le
\sqrt{\frac{1+\mathe^{-\varepsilon}}{1-\mathe^{-\varepsilon}}}\,
T(\eta_1*\cdots *\eta_n).
\]

Combining the two regimes proves that
\[
T(\bar\eta^{*n})\le C(\varepsilon)\,T(\eta_1*\cdots *\eta_n).
\]
Taking the infimum over all $\varepsilon>0$ with $\De{\varepsilon}<1$
proves the theorem.
\end{proof}
\begin{remark}
To obtain explicit constants,
choosing $\varepsilon=0.04439$
yields $\De{\varepsilon}\approx0.13691$
and $C(\varepsilon)\approx 6.71287$,
which is an upper estimate on $C_0$
and provides the lower bound
$c=\frac1{C_0}>0.1489
$.
\end{remark}

\subsection{Auxiliary results}

\begin{lemma}
\label{lem:rect}
For any probability measures $\mu,\nu$ on 
a measurable space
$(\Omega,\mathcal F)$ and any $n\ge 1$,
\[
\TV(\mu^{\otimes n},\nu^{\otimes n})
=
\sup_{\mathcal E}
\TV(\mu_{\mathcal E}^{\otimes n},\nu_{\mathcal E}^{\otimes n}),
\]
where the supremum is over all finite $\mathcal F$-measurable partitions
$\mathcal E=\{E_1,\dots,E_m\}$ of $\Omega$, and
$\mu_{\mathcal E}(j)=\mu(E_j)$, $\nu_{\mathcal E}(j)=\nu(E_j)$.
\end{lemma}
\begin{proof}
Let $\lambda:=\mu^{\otimes n}-\nu^{\otimes n}$ and
$\tau:=\mu^{\otimes n}+\nu^{\otimes n}$.
For a finite partition $\mathcal E$, write $\mathcal G_{\mathcal E}:=\sigma(\mathcal E)$ and set
\[
\mathcal A:=\bigcup_{\mathcal E}\mathcal G_{\mathcal E}^{\otimes n}.
\]
Because common refinements of finite partitions are still finite, $\mathcal A$ is an algebra.
Moreover, $\mathcal A$ generates $\mathcal F^{\otimes n}$, since it contains every measurable
rectangle $A_1\times\cdots\times A_n$ (take $\mathcal E$ refining the finite family
$\{A_1,A_1^c,\dots,A_n,A_n^c\}$).

Hence, by the finite-measure approximation theorem for algebras
\cite[\S13, Theorem~D]{halmos1974measure}, for every
$B\in\mathcal F^{\otimes n}$ and every $\varepsilon>0$ there exists
$A\in\mathcal A$ such that $\tau(A\triangle B)<\varepsilon$.
Since $|\lambda|\le \tau$,
\[
|\lambda(B)-\lambda(A)|
\le |\lambda|(A\triangle B)
\le \tau(A\triangle B)
<\varepsilon.
\]
Therefore
\[
\TV(\mu^{\otimes n},\nu^{\otimes n})
=
\sup_{B\in\mathcal F^{\otimes n}} |\lambda(B)|
=
\sup_{A\in\mathcal A} |\lambda(A)|.
\]

Now fix $\mathcal E=\{E_1,\dots,E_m\}$ and let
$\pi_{\mathcal E}:\Omega\to[m]$ be the quotient map $\pi_{\mathcal E}(x)=j$ on $E_j$.
If $A\in\mathcal G_{\mathcal E}^{\otimes n}$, then
$A=(\pi_{\mathcal E}^{\otimes n})^{-1}(C)$ for some $C\subseteq [m]^n$, and hence
\[
|\lambda(A)|
=
\bigl|\mu_{\mathcal E}^{\otimes n}(C)-\nu_{\mathcal E}^{\otimes n}(C)\bigr|.
\]
Taking the supremum over $C\subseteq[m]^n$ gives
\[
\sup_{A\in\mathcal G_{\mathcal E}^{\otimes n}} |\lambda(A)|
=
\TV(\mu_{\mathcal E}^{\otimes n},\nu_{\mathcal E}^{\otimes n}),
\]
and the lemma follows.
\end{proof}

\begin{lemma}
\label{lem:admiss-closure}
If $P,Q$
are
strictly positive probability mass functions
on a finite set $\Omega$,
then the positive measure $\eta$
defined in \eqref{eq:etai}
(i.e., 
$
\eta =
\sum_{\omega\in\Omega}
\sqrt{P(\omega)Q(\omega)}\,
\Dirac{\frac12\log\frac{P(\omega)}{Q(\omega)}}
$)
satisfies \eqref{eq:admiss}
(i.e.,
$
\int_{\R} \mathe^x\,\eta(\dd x)
=
\int_{\R} \mathe^{-x}\,\eta(\dd x)
=1
$).
Moreover, the set of
{\em admissible}
measures on $\R$
(i.e., 
finitely supported positive measures
satisfying \eqref{eq:admiss})
is closed under 
finite convex combinations
and convolution.
\end{lemma}
\begin{proof}
For the first claim, we compute
\[
\int \mathe^x\,\eta(\dd x)
=
\sum_{\omega\in\Omega}
\mathe^{\frac12
\log\frac{P(\omega)}{Q(\omega)}
}
\sqrt{P(\omega)Q(\omega)}
=
\sum_{\omega\in\Omega} P(\omega)=1,
\]
and similarly
$
\int \mathe^{-x}\,\eta(\dd x)=\sum_{\omega\in\Omega} Q(\omega)=1,
$
so $\eta$ is admissible. 

Now
if each of $\eta_1,\ldots,\eta_n$
is admissible and $a_i\ge0$, $\sum_{i=1}^n a_i=1$,
then
$\bar\eta=\sum_i a_i\eta_i$
is also admissible, since
\beq
\int \mathe^{\pm x}\,\bar\eta(\dd x)
=
\sum_{i=1}^n 
a_i
\int \mathe^{\pm x}\,\eta_i(\dd x)
=
1.
\eeq
Finally,
if $\mu$ and $\nu$ are admissible, then
\[
\int \mathe^{\pm x}\,(\mu*\nu)(\dd x)
=
\iint \mathe^{\pm(x+y)}\,\mu(\dd x)\nu(\dd y)
=
\left(\int \mathe^{\pm x}\,\mu(\dd x)\right)
\left(\int \mathe^{\pm y}\,\nu(\dd y)\right)
=
1,
\]
which proves the closure claim.

\end{proof}

\begin{lemma}
\label{lem:mult}
Let $P,Q$ be probability mass functions on $[m]$, and
define
\[
\mathcal N_{n,m}:=
\left\{
\vec k=(k_1,\dots,k_m)\in\mathbb N_0^m:\ \sum_{j=1}^m k_j=n
\right\}.
\]
Let $N:[m]^n\to \mathcal N_{n,m}$ be the count map
\[
N_j(x_1,\dots,x_n):=\sum_{t=1}^n \mathbf 1\{x_t=j\},
\qquad j\in[m].
\]
Then
\[
N_{\#}( P^{\otimes n})=\Mult(n, P),
\qquad
N_{\#}( Q^{\otimes n})=\Mult(n, Q),
\]
and
\[
\TV( P^{\otimes n}, Q^{\otimes n})
=
\TV(\Mult(n, P),\Mult(n, Q)).
\]
\end{lemma}

\begin{proof}
This result
is a standard consequence of the fact that the count map is a sufficient statistic for
discrete distributions
and that sufficient statistics preserve total
variation; see, e.g., \cite[Theorem~2]{ArratiaTavare1994}
and the discussion immediately following it. We give a short proof for completeness.

The identities
\[
N_{\#}( P^{\otimes n})=\Mult(n, P),
\qquad
N_{\#}( Q^{\otimes n})=\Mult(n, Q)
\]
are the standard multinomial count representation.

Fix $\vec k=(k_1,\dots,k_m)\in\mathcal N_{n,m}$. The fiber
$N^{-1}(\vec k)$ has cardinality
\[
\binom{n}{k_1,\dots,k_m}:=\frac{n!}{k_1!\cdots k_m!},
\]
and for every $\vec x=(x_1,\dots,x_n)\in N^{-1}(\vec k)$ we have
\[
 P^{\otimes n}(\vec x)=\prod_{j=1}^m  P(j)^{k_j},
\qquad
 Q^{\otimes n}(\vec x)=\prod_{j=1}^m  Q(j)^{k_j}.
\]
Hence
\[
\Mult(n, P)(\vec k)
=
\binom{n}{k_1,\dots,k_m}\prod_{j=1}^m  P(j)^{k_j},
\qquad
\Mult(n, Q)(\vec k)
=
\binom{n}{k_1,\dots,k_m}\prod_{j=1}^m  Q(j)^{k_j}.
\]
Therefore,
\begin{align*}
2\,\TV( P^{\otimes n}, Q^{\otimes n})
&=
\sum_{\vec x\in[m]^n}
\abs{ P^{\otimes n}(\vec x)- Q^{\otimes n}(\vec x)}
\\
&=
\sum_{\vec k\in\mathcal N_{n,m}}
\sum_{\vec x\in N^{-1}(\vec k)}
\abs{ P^{\otimes n}(\vec x)- Q^{\otimes n}(\vec x)}
\\
&=
\sum_{\vec k\in\mathcal N_{n,m}}
\binom{n}{k_1,\dots,k_m}
\abs{
\prod_{j=1}^m  P(j)^{k_j}
-
\prod_{j=1}^m  Q(j)^{k_j}}
\\
&=
\sum_{\vec k\in\mathcal N_{n,m}}
\abs{\Mult(n, P)(\vec k)-\Mult(n, Q)(\vec k)}
\\
&=
2\,\TV(\Mult(n, P),\Mult(n, Q)).
\end{align*}
\end{proof}


\end{document}